\documentclass[a4paper,10pt]{amsart}

\usepackage{amsmath,amssymb,amsthm,a4wide,tikz}

\newtheorem{theorem}{Theorem}
\newtheorem*{thm}{Theorem}

\newcommand{\diam}{\operatorname{diam}}

\newcommand{\RR}{\mathbb{R}}

\begin{document}

\title[Sharp Poincar\'{e} inequalities and hypersurface cuts]{Sharp $L^{1}-$Poincar\'{e} inequalities correspond\\ to optimal hypersurface cuts}
\author{Stefan Steinerberger}
%\thanks{\\
%\textbf{Keywords:} Poincare inequality, sharp constant, hypersurface cuts, convex bodies, isoperimetry\\
 %\textbf{MSC (2010):} 46E35, 52A38 , 52A40 }
\address{Department of Mathematics, Yale University, New Haven, CT 06511, US}
\email{stefan.steinerberger@yale.edu} 
\date{}

\begin{abstract} 
Let $\Omega \subset \mathbb{R}^n$ be convex. If $u: \Omega \rightarrow \mathbb{R}$ has mean 0, then we have the classical Poincar\'{e} inequality
$$ \|u \|_{L^p} \leq c_p \diam(\Omega) \| \nabla u \|_{L^p}$$
with sharp constants $c_2 = 1/\pi$ (Payne \& Weinberger, 1960) and $c_1 = 1/2$ (Acosta \& Duran, 2005) independent of the dimension. The sharp constants
$c_p$ for $1 < p < 2$ have recently been found by Valtorta (2012) and Ferone, Nitsch \& Trombetti (2012). The purpose of this short paper is to prove a much stronger inequality in the endpoint
$L^1$: we combine results of Cianchi and Kannan, Lov\'{a}sz \& Simonovits to show that
$$\left\|u\right\|_{L^{1}(\Omega)} \leq
\frac{2}{\log{2}} M_{}(\Omega) \left\|\nabla u\right\|_{L^{1}(\Omega)}$$
where $M_{}(\Omega)$ is the average distance between a point in $\Omega$ and the center of gravity of $\Omega$. If $\Omega$ is a regular simplex in $\mathbb{R}^n$, this yields
an improvement by a factor of $\sim \sqrt{n}$. By interpolation, this implies that that for every convex $\Omega \subset \mathbb{R}^n$ and every $u:\Omega \rightarrow \mathbb{R}$ with mean 0
$$ \|u \|_{L^p} \leq \left(\frac{2}{\log{2}} M_{}(\Omega) \right)^{\frac{1}{p}} \diam(\Omega)^{1-\frac{1}{p}} \| \nabla u \|_{L^p}.$$
\end{abstract}
\maketitle
\vspace{-10pt}

\section{Introduction} Poincar\'{e} inequalities embody the principle that in order for a function to be large, it has to grow. The question of how to exclude
constants can be handled in different ways: there is the inequality
$$\left\|u\right\|_{L^{p}(\Omega)} \leq C_{1}(\Omega,p)\left\|\nabla u\right\|_{L^{p}(\Omega)}$$
for functions $u$ satisfying Dirichlet conditions. Another way to exclude constants is to enforce vanishing mean on the function and to consider inequalities of the type
$$\left\|u-\frac{1}{|\Omega|}\int_{\Omega}{u(z)dz}\right\|_{L^{p}(\Omega)} \leq C_{2}(\Omega,p)\left\|\nabla u\right\|_{L^{p}(\Omega)}.$$
There is a large body of work on which properties $\Omega$ needs to satisfy for such an inequality to hold as well as
how the constant depends on various geometric quantities. The case $p=1$ is much simpler than the general case: by linearity, one can rearrange the 
gradient along level lines as desired using the coarea formula. It is then natural to concentrate all growth on a
level line that is relatively short while still somehow managing a balanced cut. This easy observation seems to have first been made by
Yau \cite{yau} and, independently, by Cianchi \cite{Cianchi} but only appeared very implicitely in his work (recently more explicitely in \cite{cianchi2}). 

\begin{thm}[Cianchi, \cite{Cianchi}] Let $\Omega \subset \RR^{n}$ be open, bounded and connected. Then we have the sharp inequality
$$\left\|u-\frac{1}{|\Omega|}\int_{\Omega}{u(z)dz}\right\|_{L^{1}(\Omega)} \leq
 \left(\sup_{E \subset \Omega}{\frac{2}{\mathcal{H}^{n-1}(E)}\frac{|S||\Omega \setminus S|}{|\Omega|}}\right)\left\|\nabla u\right\|_{L^{1}(\Omega)},$$
where $E \subset \Omega$ ranges over all surfaces which divide $\Omega$ into two connected subsets $S$ and $\Omega \setminus S$. 
\end{thm}
It is easy to see that the constant cannot be improved: take a hypersurface cut realizing the infimum, take a function to be constant
on both parts (with the constants chosen such that the arising function has mean zero) and mollify. We will describe a 'geometry-free'
version of this result in the last section of the paper. We should emphasize that the proof is quite simple and requires merely the coarea formula. It would also be possible to adapt 
the argument of Lefton \& Wei \cite{lefton} for the case of functions
vanishing on the boundary (see also Kawohl \& Fridman \cite{kaw}): they prove that for $f \in H^1_0(\Omega)$
$$ \|u\|_{L^1} \leq \left( \sup_{D \subset \Omega}{\frac{|D|}{|\partial D|}} \right) \|\nabla u\|_{L^1},$$
where $D$ ranges over all subsets, where $\partial D \cap \partial \Omega = \emptyset$. The underlying philosophical 
insight into the structure of these statements dates back at least to the pioneering work of Cheeger \cite{cheeger}.
The purpose of this paper is to point out some implications of Cianchi's formula, which are unknown and improve inequalities that were hithereto considered sharp.
We also condense the underlying insight in 
the form of sharp weighted $L^1-$Poincare inequalities for the unit interval.  

\section{The Poincar\'{e} constant of convex $\Omega$}
Let us restrict ourselves to the case of a convex domain $\Omega \subset \mathbb{R}^n$. 
We first mention a sharp inequality due to Acosta \& Dur\'{a}n \cite{acosta}, which is the $L^1-$analogue of the famous
 Payne-Weinberger inequality \cite{payne} dealing with $p=2$. By now, this inequality is very widely quoted in the study of 
finite element approximations. Recently, Valtorta \cite{val} as well as Ferone, Nitsch \& Trombetti \cite{ferone} 
gave the optimal constants for all $1 < p < 2$. 

\begin{thm}[Acosta \& Dur\'{a}n] Let $\Omega \subset \mathbb{R}^n$ be convex. Then
$$\left\|u-\frac{1}{|\Omega|}\int_{\Omega}{u(z)dz}\right\|_{L^{1}(\Omega)} \leq
\frac{\diam(\Omega)}{2} \left\|\nabla u\right\|_{L^{1}(\Omega)}$$
and the constant $1/2$ cannot be improved.
\end{thm}
The inequality is easily seen to be sharp by considering the unit interval or, by extension, long, thin cylinders  in arbitrary space dimension:
let $\Omega = [0,1] \times [0, \varepsilon]$ for $\varepsilon \rightarrow 0$ and consider the function $u$
as a mollification of
$$ (\chi_{[0, 1/2]} - \chi_{[1/2 , 1]})(x),$$
where $\chi$ denotes a characteristic function.  Then, as the degree of mollification and $\varepsilon$ tend to 0, a simple calculation
shows that
$$ \left\|u-\frac{1}{|\Omega|}\int_{\Omega}{u(z)dz}\right\|_{L^{1}(\Omega)} = \varepsilon$$
while
$$ \diam(\Omega) = 1 \qquad \mbox{and} \qquad   \left\|\nabla u\right\|_{L^{1}(\Omega)} = 2\varepsilon.$$

\subsection{An improvement}  While the Acosta-Dur\'{a}n inequality is sharp, we will now show that the diameter
is not the sharp geometric quantity: the inequality remains true if we replace the diameter by the average distance 
to the center of mass: this new quantity can be an arbitrary factor smaller than
the diameter if the convex set is in sufficiently high dimensions.
\begin{theorem} Let $\Omega \subset \mathbb{R}^n$ be a convex domain. Then
$$\left\|u-\frac{1}{|\Omega|}\int_{\Omega}{u(z)dz}\right\|_{L^{1}(\Omega)} \leq
\frac{2}{\log{2}} M_{}(\Omega) \left\|\nabla u\right\|_{L^{1}(\Omega)}$$
where $M_{}(\Omega)$ is the average distance between a point in $\Omega$ and the center of gravity of $\Omega$.
\end{theorem}
\begin{proof} Cianchi's theorem states that
$$\left\|u-\frac{1}{|\Omega|}\int_{\Omega}{u(z)dz}\right\|_{L^{1}(\Omega)} \leq
 \left(\sup_{E \subset \Omega}{\frac{2}{\mathcal{H}^{n-1}(E)}\frac{|S||\Omega \setminus S|}{|\Omega|}}\right)\left\|\nabla u\right\|_{L^{1}(\Omega)}.$$
The relevant isoperimetric quantity has been independently studied by Kannan, Lov\'{a}sz \& Simonovits \cite{kannan}, who prove that
 $$\sup_{E}{\frac{|S||\Omega \setminus S|)}{\mathcal{H}^{n-1}(E)|\Omega|}} \leq \frac{M_{}(\Omega)}{\log{2}},$$
where $M_{}(\Omega)$ is the average distance between a point in $\Omega$ and the center of gravity of $\Omega$.
\end{proof}
 
The average distance between a point and the center of gravity is trivially bounded by the diameter, however,
if we consider the regular simplex $P$ with comprised of $n+1$ points in $\mathbb{R}^n$
with diameter 1, then using a classical result (see e.g. Conway \& Sloane \cite[Chapter 21 2.E]{conw})$$
 \frac{1}{|P|}\int_{P}{\|x\|dx} \leq  \left(\frac{1}{|P|}\int_{P}{\|x\|^2dx} \right)^{\frac{1}{2}} = \sqrt{\frac{n}{2(n+1)(n+2)}} \diam(P).$$
This yields an improvement by a factor of $ \sim n^{-1/2}$ in $n$ dimensions.

\subsection{A general improvement.} We can now combine our result with a trivial consequence of the mean-value theorem:
we have that for $u \in C^1(\Omega)$ with vanishing mean
\begin{align*}
\left\|u\right\|_{L^{1}(\Omega)} &\leq \frac{2}{\log{2}} M_{}(\Omega) \left\|\nabla u\right\|_{L^{1}(\Omega)} \\
\left\|u\right\|_{L^{\infty}(\Omega)} &\leq \diam(\Omega) \left\|\nabla u\right\|_{L^{\infty}(\Omega)},
\end{align*}
where the second inequality would also be valid for any function that vanishes at some point inside of $\Omega$.
 Riesz-Thorin interpolation implies the following improvement for $L^p$ with $1 < p < \infty$.
\begin{theorem} Let $\Omega \subset \mathbb{R}^n$ be a convex domain. Then
$$\left\|u-\frac{1}{|\Omega|}\int_{\Omega}{u(z)dz}\right\|_{L^{p}(\Omega)} \leq
\left(\frac{2}{\log{2}} M_{}(\Omega)\right)^{\frac{1}{p}} 
\left(\diam(\Omega)\right)^{1-\frac{1}{p}} \left\|\nabla u\right\|_{L^{p}(\Omega)}$$
where $M_{}(\Omega)$ is the average distance between a point in $\Omega$ and the center of gravity of $\Omega$.
\end{theorem}
This improves the classical inequality with a constant depending on the diameter. The powers on the geometric quantities are optimal in the endpoints
in $p=1$ and $p=\infty$ (where the regular simplex in $\mathbb{R}^n$ with diameter 1 has a $L^{\infty}-$Poincar\'{e} constant that does not decay as the dimension increases while the above computation shows that $M(\Omega)$ tends to 0); it
seems reasonable to conjecture that they are also optimal in the intermediate range but this we do not know.

\section{Hypersurface cuts} 
Given any subset $S$ of a convex body $\Omega$, it is intuitively clear that
the boundary $\partial S$ cannot be too small unless $S$ itself is either very small or contains almost all of $\Omega$
(in which case $\Omega \setminus S$ is small).
The problem of quantifying this simple notion
was formulated in 1989 by Dyer, Frieze \& Kannan \cite{kan}, who also conjectured a lower bound. A first result
states that for convex $\Omega$ and any subset $S \subseteq \Omega$
$$\mathcal{H}^{n-1}\left(\partial S \cap \Omega\right) \geq \frac{1}{\diam(\Omega)}\min{(|S|, |\Omega \setminus S|)}$$
and was, using different methods, independently shown by Lov\'{a}sz \& Simonovits \cite{simon} and Karzanov \& Khachiyan \cite{kar}. 
Note the similarity to Polya's longest shortest fence problem (recently solved by Esposito, Ferone, Kawohl, Nitsch \& Trombetti \cite{esp}).

\begin{center}
\begin{figure}[h!]
\centering
\begin{tikzpicture}[scale=1]
\draw [ultra thick] (0,0) circle (1.5cm);
\draw [thick] (1.5,0) to[out=180,in=300] (-0.707*1.5,0.707*1.5);
\node at (0.4,0.7) {$S$};
\node at (0.1,-0.7) {$\Omega \setminus S$};

\draw [ultra thick] (6,0) ellipse (70pt and 15pt);
\draw [thick] (6.5,-0.5) to[out=100,in=260] (6.5,0.5);
\node at (7,0) {$S$};
\node at (5,0) {$\Omega \setminus S$};
\end{tikzpicture}
\caption{Seperating a convex set into two sets of comparable volume requires a large hypersurface unless $\Omega$ is close to a thin rectangle.}
\end{figure}
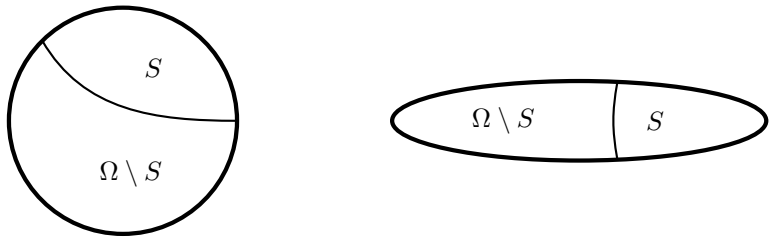
\end{center}
\vspace{-10pt}
This statement precludes the existence of bottlenecks -- it is of utmost importance in proving that random-walk based approximation
algorithms (which could be trapped in some regions if a bottleneck was present) can compute the volume of convex 
bodies in polynomial time; for their seminal work, Dyer, Frieze \& Kannan were awarded the Fulkerson prize in 1991.
The sharp constant in the statement was given by Dyer \& Frieze \cite{dyer}, sharper general bounds are given in a paper of Gromov \& Milman \cite{gro}.

\begin{thm}[Dyer \& Frieze] Let $\Omega \subset \mathbb{R}^n$ be a convex
domain and $S \subset \Omega$. Then
$$\mathcal{H}^{n-1}\left(\partial S \cap \Omega\right) \geq \frac{2}{\diam(\Omega)}\min{(|S|, |\Omega \setminus S|)}$$
and the constant 2 is optimal.
\end{thm}

The statement is again seen to be sharp by considering long, thin cylinders, where it is optimal to take $S$ to be
one half of the cylinder. If we let again $\Omega = [0,1] \times [0, \varepsilon]$ and take $S = (0, 1/2) \times (0, \varepsilon)$, then
 $$ \mathcal{H}^{1}\left(\partial S \cap \Omega\right) = \varepsilon \quad \mbox{and} \quad  \min{(|S|, |\Omega \setminus S|)} = 1/2.$$
Since $\varepsilon$ can be taken to be arbitrarily small, we can make the diameter to be as close to 1 as we wish and we see
again that at this level of generality the constant 2 cannot be improved.

\subsection{An improvement} We note that the Dyer-Frieze inequality cannot be sharp if $|S| \neq |\Omega|/2$ by proving a stronger
inequality. This is made possible by advantageously changing the 
algebraic structure of the bound. Equality is again attained for cylinders but thanks to the different algebraic structure we gain up to a factor of 2 as soon as $|S|/|\Omega|$ is close to either 0 or 1.

\begin{theorem}  Let $\Omega$ be convex and $S \subseteq \Omega$. Then
$$\mathcal{H}^{n-1}\left(\partial S \cap \Omega\right) \geq \frac{4}{\diam(\Omega)}\frac{|S||\Omega \setminus S|}{|\Omega|}$$
and the constant 4 is optimal.
\end{theorem}
Note that we can rewrite the right-hand side as
$$ \frac{4}{\diam(\Omega)}\frac{|S||\Omega \setminus S|}{|\Omega|} = \underbrace{\frac{2}{\diam(\Omega)}\min{(|S|, |\Omega \setminus S|)}}_{\mbox{Dyer-Frieze}} \cdot
\underbrace{\frac{2\max{(|S|, |\Omega \setminus S|)}}{|\Omega|}}_{\geq 1}.$$
\begin{proof}
 We identify the boundary $\partial S \cap \Omega$ with a hypersurface seperating $\Omega$ in two parts $S$ and $\Omega \setminus S$. 
Then, using the results of Acosta-Dur\'{a}n and Cianchi, we have
$$ \frac{d}{2} \geq \sup{\frac{\| f \|_{L^1}}{\| \nabla f\|_{L^1}}} = 
\sup_{E \subset \Omega}{\frac{2}{\mathcal{H}^{n-1}(E)}\frac{|S||\Omega \setminus S|}{|\Omega|}} 
\geq \frac{2}{\mathcal{H}^{n-1}(\partial S \cap \Omega)}\frac{|S||\Omega \setminus S|}{|\Omega|}.$$
Rearranging the terms yields the statement.
\end{proof}

\section{Convex geometry.}  There is a another type of result in the literature: while the statements on hypersurface cuts
were formulated in terms of merely $n-$ and $(n-1)-$dimensional measure, one could also consider other geometric
quantities such as the diameter. The direct study of the isoperimetric coefficient 
$$\sup_{E}{\frac{|S||\Omega \setminus S|}{\mathcal{H}^{n-1}(E)}},$$
where $E$ ranges over all $C^{\infty}-$surfaces dividing $\Omega$ into two disjoint parts $S$ abd $\Omega \setminus S$ 
seems to have been initiated by Bokowski \& Sperner \cite{bokowski}. Bokowski \cite{bokowski2} proved, for example,
$$\sup_{E}{\frac{|S||\Omega \setminus S|}{\mathcal{H}^{n-1}(E)}} \leq \left(1-\frac{1}{2^n}\right)\frac{(n-1)}{n(n+1)}\omega_{n-1}\diam(\Omega)^{n+1},$$
where $\omega_{n-1}$ is the volume of the unit $(n-1)-$ball. Further results are due to Santal\'{o} \cite{santalo}, Gysin \cite{gysin} and Mao \cite{mao}. 
\subsection{An improvement.} We are able to complement existing results on hypersurface cuts in a sharp form: the extremal example is again given by long, thin cylinders. This result is generally weaker than 
the inequality of Kannan, Lov\'{a}sz \& Simonovits, we note it here merely because its constant is sharp and it
improves statements from the third line of research (i.e. Bokowski \& Sperner \cite{bokowski}, Bokowski \cite{bokowski2}, ...).
\begin{theorem} Let $\Omega \subset \RR^n$ be convex and let $E \subset \Omega$ be a separating hypersurface splitting $\Omega$ into $S$ and $\Omega \setminus S$. Then we have
 $$\sup_{E}{\frac{|S||\Omega \setminus S|}{\mathcal{H}^{n-1}(E)}} \leq \frac{\diam(\Omega)}{4}|\Omega|,$$
where the constant $1/4$ cannot be improved.
\end{theorem}
\begin{proof}
The result follows again from Cianchi's equivalence and the inequality of Acosta-Dur\'{a}n. Considering again
a thin rectangle, taking $S$ to be half of the rectangle and letting the rectangle collapse to a line, we see that
the constant is sharp.
\end{proof}

\section{Weighted $L^1-$Poincar\'{e} inequalities on the unit interval}

The purpose of this section is to state condense the relevant insight into a geometry-free form: this gives sharp weighted
$L^1-$Poincar\'{e} inequalities on the unit interval for an arbitrary nonvanishing weight: let $\nu:[0,1] \rightarrow \mathbb{R}_{+}$ be a nonvanishing weight $\nu(x) > 0$ on the closed unit interval. We are interested in
sharp Poincar\'{e} constants for the inequality
$$ \int_{0}^{1}{|f(x)|\nu(x) dx} \leq C\int_{0}^{1}{|\nabla f(x)|\nu(x) dx}.$$
We need to exclude constant functions and, as before, we can do so prescribing either Dirichlet conditions or vanishing mean value.

\begin{theorem} Let $\nu:[0,1] \rightarrow \mathbb{R}_{+}$ be a nonvanishing continuous weight on the closed unit interval.
 Let $f:[0,1] \rightarrow \mathbb{R}$ be once differentiable and satisfy $f(0) = 0$. Then
$$ \int_{0}^{1}{|f(x)|\nu(x) dx} \leq \left(\max_{0 \leq x \leq 1}{\frac{1}{\nu(x)}\int_{x}^{1}{\nu(z)dz}}\right)\int_{0}^{1}{|f'(x)|\nu(x) dx}$$
and the constant is sharp.
\end{theorem}
\begin{proof}[Proof of Theorem 4.] Let $f$ be given. We consider a new function
$$ g(x) := \sup_{0 \leq z \leq x}{|f(z)|}.$$
Clearly $g(0) = 0$. Furthermore $|g(x)| \geq |f(x)|$, $g$ is differentiable a.e. and $|g'(x)| \leq |f'(x)|$. It thus suffices to prove the bound for a monotonically increasing function $g(x)$. 
Since $g(x)$ is monotonically increasing, we will now re-interpret the weighted derivative as a Riemann-Stieltjes integral (see e.g. \cite{gordon})
$$ \int_{0}^{1}{g'(x)\nu(x) dx} = \int_{0}^{1}{\nu(x) dg(x)}$$
and prove the inequality for this generalized notion. All subsequent
arguments could be equally carried out by taking limits of mollifications (which is more cumbersome). Furthermore, note that
equality will not usually not be assumed for $C^1$-functions because generically the extremizer is merely in BV and has a discontinuity. The
remainder of the argument makes crucial use of linearity of the statement: 
suppose 
$$ \int_{0}^{1}{g(x)\nu(x) dx} \leq \left(\max_{0 \leq x \leq 1}{\frac{1}{\nu(x)}\int_{x}^{1}{\nu(z)dz}}\right)  \int_{0}^{1}{\nu(x)dg(x)}$$
holds for a set $\mathcal{A}$ of monotonically increasing functions vanishing at the origin. If $h_1, h_2 \in A$, then
$$ a h_1 + b h_2 \in \mathcal{A} \qquad \mbox{for all} \quad a, b \geq 0.$$
We will prove the inequality for translations of the heaviside function; for any $0 < z < 1$ let
 $$ H_{z}(x) = \begin{cases}
            0 \qquad &\mbox{if}~x \leq z \\
            1 \qquad &\mbox{if}~x > z.
           \end{cases}$$
By definition, we have
$$ \int_{0}^{1}{H_{z}(x)\nu(x) dx} = \int_{z}^{1}{\nu(x) dx} \qquad \mbox{and} \qquad  \int_{0}^{1}{\nu(x)dH_z(x)} = \nu(z)$$
from which the validity of the inequality for these functions follows immediately. However, we can now write $g(x)$ as a linear superposition of shifted 
Heaviside functions 
$$ g(\cdot) = \int_{0}^{1}{H_z(\cdot) d g(z)}$$
and linearity implies the result.
\end{proof}
The proof actually shows more: it shows that 'near-extremizers' can be characterized: they can only increase at places, where the isoperimetric ratio
is 'large' i.e. on
$$ A = \left\{z \in [0,1]: \frac{1}{\nu(z)}\int_{z}^{1}{\nu(y)dy} \geq (1-\varepsilon)\max_{0 \leq x \leq 1}{\frac{1}{\nu(x)}\int_{x}^{1}{\nu(z)dz}}\right\}.$$
Note that $A$ is always non-empty because of compactness. If $A = \left\{x_0\right\}$ contains only one element, then there exists a unique
extremizer to the inequality, which is given by $H_{x_0}$.

\begin{theorem} Let $\nu:[0,1] \rightarrow \mathbb{R}_{+}$ be a nonvanishing continuous weight on the closed unit interval.
 Let $f:[0,1] \rightarrow \mathbb{R}$ be once differentiable and have vanishing weighted mean value
$$ \int_{0}^{1}{f(x)\nu(x)dx} = 0.$$
 Then
$$ \int_{0}^{1}{|f(x)|\nu(x) dx} \leq \left(\max_{0 \leq x \leq 1}{\frac{2}{\nu(x)}\frac{\left(\int_{0}^{x}{\nu(z)dz}\right)\left(\int_{x}^{1}{\nu(z)dz}\right)}{\int_{0}^{1}{\nu(z)dz}}}\right)
\int_{0}^{1}{|\nabla f(x)|\nu(x) dx}$$
and the constant is sharp.
\end{theorem}
This statement can be proven in a similar way as Theorem 4 (or even be deduced from Theorem 4); we leave it to the reader as an exercise.\\

\textbf{Acknowledgments.} I am very grateful to Andrea Cianchi, Isaac Chavel, Laszlo Lov\'{a}sz and Frank Morgan for valuable 
contributions to the history of the problems and to Florian Pausinger for a critical reading of the manuscript. 
I am especially grateful to Bernd Kawohl for his continued efforts and for sharing his extensive knowledge on these matters.

\end{document}